\documentclass[12pt]{article}
\usepackage{graphicx}
\usepackage{verbatim}

\usepackage{amssymb,amsmath}
\usepackage{amsthm}
\usepackage{epsfig}

\newtheorem{corollary}{Corollary}[section]
\newtheorem{lemma}[corollary]{Lemma}
\newtheorem{proposition}[corollary]{Proposition}
\newtheorem{theorem}[corollary]{Theorem}
\newcommand{\Prob} {{\bf P}}
\newcommand{\Z}{{\mathbb Z}}

\newcommand{\C}{{\mathbb C}}

\newcommand{\Q}{{\bf Q}}

\def \reff#1{(\ref{#1})}

\def \ra  {\rightarrow}

\def \Im {{\rm Im}}
\def \Re {{\rm Re}}

\def \reals {{\mathbb R}}
\def \integers {{\mathbb Z}}
\def \Half {{\mathbb H}}
\def \Halfbar{{\overline{\mathbb{H}}}}

\def \strip{{\cal S}}

\def \walk {{\mathcal W}}
\def \hwalk {{\mathcal H}}
\def \bridge{{\mathcal B}}
\def \ibridge{{\mathcal I}}

\def \acat {{\oplus}}
\def \bcat {{\otimes}}

\def \scaleexp{{\rho}}

\newenvironment{remark}[1][Remark]{\begin{trivlist}
\item[\hskip \labelsep {\bfseries #1}]}{\end{trivlist}}
\newenvironment{remarks}[1][Remarks]{\begin{trivlist}
\item[\hskip \labelsep {\bfseries #1}]}{\end{trivlist}}
\newenvironment{definition}[1][Definition]{\begin{trivlist}
\item[\hskip \labelsep {\bfseries #1}]}{\end{trivlist}}
\newenvironment{definitions}[1][Definitions]{\begin{trivlist}
\item[\hskip \labelsep {\bfseries #1}]}{\end{trivlist}}

\begin{document}

\title{The self-avoiding walk spanning a strip}

\author {
Ben Dyhr \footnote{Metropolitan State College of Denver}
\and
Michael Gilbert \footnote{University of Arizona}
\and 
Tom Kennedy \footnote{ University of Arizona; research supported  
by NSF grant DMS-0758649.}
\and
Gregory F. Lawler\footnote {University of Chicago; 
research supported by NSF grant DMS-0907143.}
\and 
Shane Passon \footnote{University of Arizona}
}

\date{}

\maketitle

\begin{abstract}
We review the existence of the infinite length self-avoiding walk in 
the half plane and its relationship to bridges. We prove that this 
probability measure is also given by the limit as 
$\beta \rightarrow \beta_c-$ 
of the probability measure on all finite length walks $\omega$ 
with the probability of $\omega$ proportional to $\beta^{|\omega|}$ 
where $|\omega|$ is the number of steps in $\omega$. ($\beta_c$ is the 
reciprocal of the connective constant.)
The self-avoiding walk in a strip $\{z : 0<\Im(z)<y\}$ is defined 
by considering all self-avoiding walks $\omega$ in the strip which start 
at the origin and end somewhere on the top boundary with 
probability proportional to $\beta_c^{|\omega|}$ 
We prove that this probability measure may be obtained by conditioning
the SAW in the half plane to have a bridge at height $y$. 
This observation is the basis for simulations to test conjectures
on the distribution of the endpoint of the SAW in a strip and the relationship
between the distribution of this strip SAW and SLE$_{8/3}$.
\end{abstract}

\newpage

\section{Introduction}

Self-avoiding walks (SAW) on the integer lattice give
a simple model of polymers in a dilute solution or
more generally random configurations whose only 
constraint is  given by a self-repulsion.  
In the SAW problem on a lattice, e.g., $\Z^d$, one considers
nearest neighbor walks on the lattice which visit 
each lattice site at most once. 
The basic principle of the standard SAW model is that all walks
of a given length (perhaps constrained to have a 
certain starting point or to stay in a certain subset) are
given the same measure.  
We will consider only $d=2$ in this paper, but the
problem is interesting for all $d \geq 2$; see
\cite{Madras} for an overview.

The classic problem for which
there are few rigorous results is to understand the
asymptotic behavior as $N$ tends to infinity of the uniform
probability measure on the set of SAW's of $N$ steps
starting at the origin.  Among the open questions is 
a proof that the {\em infinite} SAW exists.  The latter
should be the probability measure on infinite SAW's starting
at the origin such that for a finite length SAW $\omega$, 
the probability of the set of infinite walks that start with 
$\omega$ is equal to the limit
as $N \rightarrow \infty$ of the fraction of SAW's of
length $N$ that start with the walk $\omega$.
A different question is to find the scaling limit of
the infinite SAW which should be very closely related to
the scaling limit of the uniform measure.
The scaling limit should be obtained by replacing the lattice
with unit spacing by a lattice with spacing $\delta$ and letting
$\delta$ go to zero.

One can ask a similar question about SAW's restricted
to stay in a half-plane.  In this case, it can be
proved that the infinite half-plane SAW exists.  However,
it is still open to prove that it
has a scaling limit. 

It is believed that the SAW is one of the planar models that
exhibits conformal invariance in the scaling limit.  In
\cite{LSW_SAW}, see also \cite{Lawler_utah}, precise
conjectures were made about the kind of scaling limit.
We state one version of the conjecture now.  Let us write
$\Z^2 = \Z + i \Z$ and let $\strip_y = \{z \in \C: 0 < \Im(z)
 < y\}$.  For each positive integer $y$, consider the set of
SAW's starting at $0$, ending at $\{z: \Im(z) = y\}$ and
otherwise staying in $\strip_y \cap \Z^2$.  To each such
walk $\omega$ we assign weight equal to
$\beta_c^{|\omega|}$.  Here $|\omega|$ denotes the number
of steps on $\omega$ and $\beta_c$ is the inverse of
the connective constant defined roughly by saying that
the number of SAW's of length $n$ in $\Z^2$ starting at the origin  
grows like $\beta_c^{-n}$.  It is important to note that
although walks of the same length get the same measure, walks of different
length are not weighted the same.
If $x,y\in \Z$, let $Z_{\strip_y}(0,x+iy)$ denote the corresponding partition
function, that is, the sum of weights of all the walks in $\strip_y$ that
end at $x+iy$ and let $Z_{\strip_y}(0,\integers+iy) = 
\sum_{x} Z_{\strip_y}(0,x+iy)$.  
(It is not obvious that $Z_{\strip_y}(0,\integers+iy)$ is finite. 
This was proved in \cite{chayes}, and we give a proof in 
section \ref{sect_saw_strip}.)
Let $\Q^y_x$, $\Q^y$
denote the corresponding probability measures on paths obtained
by normalizing by $Z_{\strip_y}(0,x+iy), Z_{\strip_y}(0,\integers+iy),
$ respectively. Then the conjectures are:
\begin{itemize}
\item  There exists a scaling exponent $b$ and a scaling
function $\rho$ such that 
\begin{equation}  \label{intro.1}
       Z_{\strip_y}(0,x+iy) \sim \rho(x/y) \, y^{-2b} , \;\;\;\;
   y \rightarrow \infty . 
\end{equation}
\item  

If $x_y \rightarrow \infty$ such that $x_y/y
 \rightarrow x$ then the measures $\Q^y_{x_y}$ have 
a scaling limit which is a probability measure $\mu_x$
on paths in $\strip_1$ connecting $0$ and $x+i$.

\item  Similarly, the measures $\Q^y$ have a scaling limit
which is a probability measure $\mu^1$
on paths in $\strip_1$ connecting $0$ and some point on the upper 
boundary of $\strip_1$, and which can be written as
\[           \Q^y  \rightarrow \mu^1 = 
c \int_{-\infty}^\infty \rho(x) \, \mu_x \, dx.\]
 
\item  The probability measures $\mu_x$ are invariant under conformal
transformations. For example, if $f: \strip_1 \rightarrow \strip_1$
is a conformal transformation with $f(0) = 0, f(i) = x+i$, then the
image of $\mu_0$ under $f$ is $\mu_x$. (Here probability measures are
considered on paths  modulo reparametrization.)
\end{itemize}

More generally, if one has a simply connected domain $D$ and two
boundary points $z,w$, one expects to get a probability measure
$\mu_D(z,w)$ in the scaling limit, 
and that these measures will be conformally
invariant. In general, lattice corrections that persist
in the scaling limit  will require a
more complicated form for the partition function than seen
in \eqref{intro.1}, but we will not worry about this here.   

The conjecture \eqref{intro.1} suggests that we should have the 
following scaling rule for the function $\rho$. 
If $f:\strip_1 \rightarrow \strip_1$ is a conformal 
transformation with $f(0) = 0, f(i) = x+i$, then
\begin{equation}  \label{intro.2}
              \rho(0) = |f'(0)|^b \, |f'(i)|^b \, \rho(x). 
\end{equation}
We give a heuristic argument for \reff{intro.2} in section
\ref{sect_sim_part_func}.

In \cite{LSW_SAW}, it was shown that if the infinite half-plane
SAW has a conformally invariant limit then the limit must
be the chordal Schramm-Loewner evolution SLE with exponent
$\kappa = 8/3$.  Using this  and rigorous results about SLE
\cite{LSW_restriction}, predictions for critical exponents for
the SAW were rederived.  In particular, one predicts
\[       b = \frac 58 .\]
Because these predictions rely on the
assumption of the conformally invariant limit, they are 
not rigorous derivations.  This conjecture also led to new 
predictions about the behavior of SAW, e.g., the probability that
a half-plane SAW hits a slit emanating from the real line.
These new conjectures were tested numerically in \cite{Kennedya,Kennedyb},
and the simulations were very consistent with the predictions.
These simulations give more evidence to believe the assumption
of a conformally invariant limit.  The work in \cite{Kennedya,Kennedyb}
sampled from the infinite half-plane SAW (more precisely, it sampled
from the uniform measure on very long half-plane walks and then
restricted the measure to a smaller initial segment of the path), and hence
did not give a test of the scaling conjecture for a bounded domain.

The finite scaling rule is difficult to test for two reasons:
\begin{itemize}
\item  The pivot algorithm, which is the fastest method for producing
SAW's, chooses walks from the uniform measure on walks of a fixed number
of steps (perhaps restricted to the half-plane).  It seems much
harder to simulate quickly from the measure where paths are weighted
by $\beta_c^{|\omega|}.$
\item  Unless the boundary of our domain consists of straight lines
parallel to the coordinate axes, the lattice corrections in the scaling
rule \eqref{intro.1} are very difficult to understand.
\end{itemize}

In this paper, we test the conjectures by considering walks in
$\strip_y \cap \Z^2$.  The limit domain $\strip_1$ has only horizontal
edges, so one does not expect problems in the lattice corrections.  Also,
as we will show in section \ref{sect_inf_half}, sampling from $\Q^y$
is the same as sampling from the infinite half-plane SAW and conditioning
on having a ``bridge'' or ``horizontal cut'' at $y$.  Our work was
motivated by \cite{albert_hugo} who considered bridges of $SLE_{8/3}$.
Using the prediction $b = 5/8$ and the scaling rule \eqref{intro.2}, we
will see in section \ref{sect_sim_part_func}
that the predicted scaling limit for $\Q^y$ is
\[               \mu^1
=   c \, \int_{-\infty}^\infty \rho(x) \, \mu_x \, dx, \]
where
\[   \rho(x) \propto  \left[\cosh
 \left(\frac \pi 2 x\right) \right]^{-5/4}, \]
and $\mu_x$ is the probability measure associated to chordal
$SLE_{8/3}$ from $0$ to $x+i$ in $\strip_1$.  For the scaling limit
of $\Q^y$ we consider three main random variables: the distribution of
the real part of the endpoint (which should have density $c \,\rho(x)$
in the limit), the distribution of the rightmost excursion of the path,  
and the probability that the path passes to the left of a fixed
$z \in \strip_1$. The last probability was first computed for
each $\mu_x$ by Schramm \cite{Schramm2}
and one obtains it for $\mu^1$ by integrating.  We are very
pleased to find our simulations give strong numerical evidence for the
finite scaling conjectures.

This paper is divided into two parts.  The next section is theoretical
and establishes mathematically the relationship between the infinite
half-plane SAW and the measure on paths weighted by $\beta_c^{|\omega|}$.
We decided to make this section self-contained except for one difficult
estimate of Kesten.  In particular, the existence of measures on
infinite bridges and half-plane SAW's in \cite{Madras, LSW_SAW} are reproved.
We include a discussion of half-plane walks started at interior points
although we do not do simulations from this measure.
In the last section we give a heuristic argument for 
\reff{intro.2}, and then discuss the results of the simulations in detail.

\section{Infinite half-plane SAW}
\label{sect_inf_half}

\subsection{Half-plane SAW starting at the origin}

In this section we review some facts about
the infinite self-avoiding half-plane walk.
The construction of the measure is due to Madras
and Slade \cite{Madras} who called it the infinite bridge. In
\cite{LSW_SAW} it was shown  that this measure
is in fact the infinite half-plane SAW.  Since
it is not too difficult, we will give
proofs of the results relying only on a
difficult estimate  due to  Kesten.  

We write a self-avoiding walk (SAW) in $\Z^2
= \Z \times i \Z$  as
\[  \omega = [\omega_0,\ldots,\omega_n].\]
We write $|\omega| = n$ for the number of steps
of $\omega$.  We include the trivial $0$-step
walks.  Let $\walk$ denote the set of SAW's in
$\Z^2$ and $\walk_0$ the set of $\omega \in \walk$
with $\omega_0 = 0$. If $C_n = \#\{\omega \in
\walk_0: |\omega| = n\}$, then the connective constant,
which we will denote by $1/\beta_c$ is defined by
\[           C_n \approx  (1/\beta_c)^n , \;\;\;\;
  n \rightarrow \infty . \]
Equivalently, $\beta_c$ is the radius of convergence
of the power series
\[             \sum_{\omega \in \walk_0}
    \beta^{|\omega|} = \sum_{n=0}^\infty C_n \,
  \beta^n . \]
The existence of the connective constant is easily
established from the subadditivity relation
$C_{n+m} \leq C_n \, C_m$ from which one can also
conclude that $C_n \geq \beta_c^{-n}$. The value
of the constant is not known; we will only use the
simple estimate $0 < \beta_c < 1$.

We will define two different, but related, notions
of concatenation of paths.  If $\omega^1 = [\omega^1_0,
\ldots,\omega^1_m]$ and $\omega^2 = [\omega^2_0,\ldots,
\omega^2_n]$ are two SAW's with $\omega^1_m = \omega^2_0$
we define $\omega^1 \acat \omega^2$ to be the
$(m+n)$-step walk
\[   \omega^1 \acat \omega^2 =
          [\omega^1_0,\ldots,\omega^1_m,\omega^2_1,\ldots,
\omega^2_n]. \]
The concatenation $\omega^1 \acat \omega^2$ is defined only
when the terminal point of $\omega^1$ is the same
as the initial point of $\omega^2$.  If $\omega^2 \in
\walk_0$ we define $\omega^1 \bcat \omega^2$ to be
$\omega^1 \acat (\omega^1_m + \omega^2)$, i.e., 
\[ \omega^1 \bcat \omega^2 = [\omega^1_0,\ldots,\omega^1_m,
   \omega^1_m + \omega^2_1 ,\ldots,\omega^1_m + \omega^2_n]. \]
We similarly define $\omega^1 \acat \cdots \acat \omega^k$,
$\omega^1 \bcat \cdots \bcat \omega^k$. 
 We write $\omega^1 \prec \omega$
if we can write $\omega = \omega^1 \acat \omega^2$ for
some $\omega^2 \in \walk$.

\begin{definitions}
$\;$

\begin{itemize}  

\item A {\em (upper) half-plane SAW
(starting at the origin)} is a walk
$\omega = [\omega_0=0,\ldots,\omega_n] \in \walk_0$ with
$\Im[\omega_j] > 0$ for $j > 0$.  We let $\hwalk$ denote the
set of half-plane walks. The {\em height} of a walk $\omega \in
\hwalk$ is defined to be the maximal value of $\Im[\omega_j],
j=0,\ldots,n$.   The trivial walk $\omega = [0]$ is
the unique $\omega \in \hwalk$ with $h(\omega) = 0$. 
$\hwalk^*$ denotes the set of walks $\omega \in \hwalk$ with 
$h(\omega) \geq 1$, i.e., the nontrivial walks in $\hwalk$.

\item  A {\em weakly half-plane SAW} is a walk
$\omega = [\omega_0=0,\ldots,\omega_n] \in \walk_0$ with
$\Im[\omega_j] \geq 0$ for $j > 0$.  We let
$\tilde \hwalk$ denote the set of weakly half-plane
walks. There is a natural bijection
between $\tilde \hwalk$ and $\hwalk^*$ given by
\[     \omega \longleftrightarrow [0,i] \bcat
  \omega . \]

\item A {\em bridge} is a walk $\omega =[\omega_0=0,\ldots,
\omega_n] \in \hwalk$ with $\Im[\omega_n] =
h(\omega)$. The trivial walk $[0]$ is a bridge.
We let $\bridge$ denote the set of bridges and
$\bridge^*$   the set of   bridges $\omega$
with $h(\omega) \geq 1$, i.e., the nontrivial bridges.

\item An {\em irreducible bridge}  is a bridge
$\omega  \in \bridge$ with $h(\omega) \geq 1$ such that
$\omega$ cannot be written as
\[            \omega = \omega^1 \bcat \omega^2 \]
with $\omega^1,\omega^2 \in \bridge^*$.  We let $\ibridge$
denote the set of irreducible bridges.

\item More generally, let $\ibridge_k$ denote the
set of bridges $\omega$ with $h(\omega) \geq
k$  such that $\omega$
cannot be written as
\[        \omega = \omega^1 \bcat \omega^2 , \]
with $\omega^1,\omega^2 \in \bridge^*$ and
$h(\omega^1) \geq k$.  In other words, $\ibridge_k$ is
the set of bridges that are ``irreducible above
height $k$''. Note that $\ibridge =
\ibridge_1$.

\end{itemize}

\end{definitions}

An important observation is that every bridge $\omega \in \bridge^*$
can be written uniquely as
\begin{equation}  \label{renew}
           \omega = \omega^1 \bcat \omega^2, \;\;\;\;
  \omega^1 \in \ibridge, \;\; \omega^2 \in \bridge . 
\end{equation} 
The bridge $\omega$ is irreducible if and only if $\omega^2$
is the trivial bridge.  More generally, if $h(\omega) \geq k$,
then $\omega$ can also be written uniquely as
\begin{equation}  \label{mar18.1}
        \omega = \omega^3 \bcat \omega^4 ,
\;\;\;\;  \omega^3 \in \ibridge_k, \;\; \omega^4 \in \bridge.
\end{equation}

We define the generating functions
\[        W(\beta) = \sum_{\omega \in  \walk_0} \beta^{|\omega|},
\;\;\;\;
       H(\beta) = \sum_{\omega \in \hwalk} \beta^{|\omega|}, \]
\[      \tilde H(\beta) =  \sum_{\omega \in \tilde \hwalk}
   \beta^{|\omega|}, \;\;\;\;
 B(\beta) = \sum_{\omega \in \bridge}
   \beta^{|\omega|}, \]
\[          I_k(\beta) = \sum_{\omega \in \ibridge_k}
      \beta^{|\omega|} , \;\;\;\;  I(\beta) = I_1(\beta). \]
The generating functions are increasing in $\beta$ for
$\beta > 0$ and for such $\beta$,
 $I_k(\beta) \leq B(\beta) \leq H(\beta) \leq W(\beta)$.
The bijection between $\tilde \hwalk$ and $\hwalk^*$ implies
\begin{equation}  \label{mar19.2}
       H(\beta) = 1 + \beta \, \tilde H(\beta). 
\end{equation}
Also, eq. \eqref{renew} gives
 \begin{equation}  \label{mar19.3}
  B(\beta) = 1 + I(\beta) \, B(\beta).
\end{equation} 
Recall that $\beta_c$ is the radius of convergence of
$W(\beta)$.  The construction of the measure
relies on the following  fact due to Kesten \cite{kestena}.

\begin{proposition} \label{kesten_prob}
 $B(\beta_c) = \infty $ and hence
 $I(\beta_c) = 1$.
\end{proposition}

\begin{proof} 
We follow the argument in \cite{Madras}
to prove $B(\beta_c) =
\infty$.
Since  $C_n \geq
(1/\beta_c)^n$,  $W(\beta_c) = \infty$.
By splitting a walk $\omega \in \walk_0$ at the
last point at which the walk achieves its minimal
imaginary part, we can see that
\[         W(\beta) \leq \tilde H(\beta) \,  H(\beta).\]
Using \eqref{mar19.2}, this gives $H(\beta_c) =
 \infty$.
Similarly, each
 $\omega \in \hwalk$ with $h(\omega) \geq 1$
can be written uniquely as
\[  \omega = \omega^1 \bcat (-\omega^2) \bcat \omega^3
     \bcat \cdots (-1)^{k-1} \omega^k , \]
where $\omega^1,\omega^2,\ldots,\omega^k \in \bridge$ with
$h(\omega^1) > h(\omega^2) > \cdots > h(\omega^k)$. (For
example, $\omega^1=[\omega_0,\ldots,\omega_j]$ where
$j$ is the largest index with $\Im[\omega_j] =
h(\omega).$)
Let 
\[     B_k(\beta) = \sum_{\omega \in \bridge,
h(\omega) = k}
              \beta^{|\omega|}.\]
Then,
\[             H(\beta) \leq \prod_{k=1}^\infty
      [1+B_k(\beta)]  \leq \exp
  \left\{\sum_{k=0}^\infty B_k(\beta)\right\}
  = e^{B(\beta)}, \;\;\;\; \beta < \beta_c. \]
Hence $B(\beta_c) = \infty$. 
The second assertion of the proposition
follows from \eqref{mar19.3} which
implies that 
\[        B(\beta) = \frac{1}{1 - I(\beta)}. \]
\end{proof}

This proposition justifies the next definition.
In the introduction we used boldface $\Q$'s to denote probability
measures on SAW's in a strip. 
$\Q^y_x$ was a probability measure on SAW's that end at a fixed point on the 
upper boundary, and $\Q^y$ was a probability measure on SAW's that end
anywhere on the upper boundary.
In the following definition and subsequent definitions
we use non-boldface $Q$'s with subscripts and superscripts to denote
probability measures that are defined on SAW's in the upper half-plane. 

\begin{definitions} $\;$

\begin{itemize}

\item  Let $Q$ denote the probability measure on
$\ibridge$ defined by   
\[          Q(\omega) = \beta_c^{|\omega|}, \;\;\;\;
   \omega \in \ibridge \]

\item  If $j \geq 1$, we consider the product space $\ibridge^j$
and define the probability measure $Q^j $ by  product measure. 
We also write $Q^j$ for the extension
to    $\hwalk$ with $Q^j[\hwalk \setminus
\ibridge^j] = 0$ 
and  for
the corresponding probability measure on $\bridge$
given by
\[   Q^j(\omega^1 \bcat \cdots \bcat \omega^j)
        = Q(\omega^1) \cdots Q(\omega^j) , \;\;\;\;
   \omega^1,\ldots,\omega^j \in \ibridge. \]
Here  $Q^j(\omega) = 0$ if $\omega \in
\bridge$ is not of the form 
$\omega^1 \bcat \cdots \bcat \omega^j, $  $
  \omega^1,\ldots,\omega^j \in \ibridge. $

\item 
  We define  $Q^\infty$ on $\ibridge^\infty =
\ibridge \times \ibridge \times \cdots$,
which can be considered as a measure on infinite paths.
 The {\em infinite half-plane SAW} is
the measure on infinite self-avoiding paths
induced by $Q^\infty$.
\end{itemize}

\end{definitions}

We have taken this to be the definition.  Perhaps we should
have defined this to be the infinite bridge as was done by
Madras and Slade, 
but we will show that this definition is equivalent
to other definitions that are more naturally called 
infinite half-plane SAW.  There are two natural ways
to define an infinite half-plane SAW: either as the limit
as $\beta \rightarrow \beta_c-$ of walks $\omega \in 
\hwalk$ weighted by $\beta^{|\omega|}$ or the limit as
$n \rightarrow \infty$ of the uniform measure on walks
$\omega \in \hwalk$ with $|\omega| = n$.  We consider
the first of these in this section  and discuss the
second in Section \ref{unifsec}.

\begin{definition}  If $\beta < \beta_c$, $\Prob_\beta$
denotes the probability measure on $\hwalk$ given by
\[   \Prob_{\beta}(\omega) = \frac{\beta^{|\omega|}}
                    {H(\beta)}, \;\;\;\;
  \omega \in\hwalk. \]
\end{definition}

The next proposition shows that in some sense the
limit of $\Prob_\beta$ as $\beta \rightarrow \beta_c-$
is the infinite half-plane SAW.

\begin{proposition} \label{mar19.prop1}
 Suppose $\omega^1,\ldots,\omega^j \in
\ibridge$.  Let $\hwalk(\omega^1,\ldots,\omega^j)$ denote
the set of $\omega \in \hwalk$ of the form
\begin{equation}  \label{form}
     \omega = \omega^1 \bcat \cdots \bcat
   \omega^j \bcat \tilde \omega , 
\end{equation}
with $\tilde \omega \in \hwalk$.  Then
if $\beta < \beta_c$,
\begin{equation}  \label{mar18.2}
     \Prob_\beta[\hwalk(\omega^1,\ldots,
   \omega^j) ] = \beta^{|\omega^1| + \cdots +
   |\omega^j|}  . 
\end{equation}
In particular,
\[  \lim_{\beta \rightarrow \beta_c-}
  \Prob_\beta[\hwalk(\omega^1,\ldots,
   \omega^j) ]
 = \beta_c^{|\omega^1| + \cdots +
   |\omega^j|}  = Q^j(\omega^1 \bcat
  \cdots \bcat \omega^j) . \]
\end{proposition}

\begin{remarks}  $\;$ 

\begin{itemize}
\item  $\Prob_\beta$ is a probability measure on finite length
walks in $\Half$, while $Q^\infty$ is a measure on infinite length
walks in $\Half$. The precise sense in which $\Prob_\beta$ converges to 
$Q^\infty$ is the following. In the proposition 
$\hwalk(\omega^1,\ldots,\omega^j)$ denotes the set of finite length
walks that start with $\omega^1 \bcat \cdots \bcat \omega^j$. If we let 
$\hwalk^\infty(\omega^1,\ldots,\omega^j)$ denote the set of infinite 
length walks that start with 
$\omega^1 \bcat \cdots \bcat \omega^j$, 
then 
\begin{equation}
Q^\infty[\hwalk^\infty(\omega^1,\ldots,\omega^j)]=
\beta_c^{|\omega^1| + \cdots + |\omega^j|} . 
\end{equation}
So the proposition says that 
\begin{equation}
 \lim_{\beta \rightarrow \beta_c-}
  \Prob_\beta[\hwalk(\omega^1,\ldots,
   \omega^j) ] =
 Q^\infty[\hwalk^\infty(\omega^1,\ldots,\omega^j)]
\end{equation}

\item  Since $Q^j$ is a probability measure,
a corollary of this proposition is the following.  For
each $j$ let $A_j$ denote the set of walks $\omega
 \in \hwalk$ that are {\em not} of the form
\[          \omega = \omega^1 \bcat \cdots \bcat
   \omega^j \bcat \tilde \omega, \;\;\;\;
   \omega^1,\ldots,\omega^j \in \ibridge, \;\;
  \tilde \omega \in \hwalk. \]
Then
\[   \Prob_\beta(A_j) = 1 - \sum_{\omega^1,\ldots,
  \omega^j \in \ibridge} \beta^{|\omega^1| +
 \cdots + |\omega^j|} , \]
and hence 
\begin{equation} \label{shaneedit1}  \lim_{\beta \rightarrow \beta_c-}
              \Prob_\beta(A_j) = 0 . \end{equation}
\end{itemize}

\end{remarks}

\begin{proof} Equation \eqref{mar18.2} follows
  from
\[     \sum_{\omega \in \hwalk(\omega^1,\ldots,
   \omega^j) } \beta^{|\omega|}
 = \beta^{|\omega^1| + \cdots +
   |\omega^j|}  \, H(\beta), \]
and the rest follows immediately from the definitions.
\end{proof}

\begin{definition}
If $\omega \in \walk$, we say that $\omega$
has a {\em (horizontal) cut at level $k$}, if
we can write
\[          \omega = \omega^1 \acat
  [j +ki, j + (k+1) i] \acat \omega^2, \]
with 
   $j \in \Z,  \omega^1,\omega^2 \in \walk$
and
$\omega^1 \subset \{x+iy: y \leq k\}$,
$\omega^2 \subset \{x+iy: y \geq k+1\}$. 
Let $\hat \walk^k$ denote the set of walks that
have a cut at level $l$ for some $l \geq k$.
\end{definition}

Note that if $j > k$ and $\omega$ is of the form
\eqref{form}, then $\omega \in \hat \walk^k$.
Since $\hwalk\setminus \hat{ \mathcal{W}^k}\subset A_k$, 
equation \reff{shaneedit1} implies for every $k \geq 1$, 
\begin{equation}  \label{mar19.8}
   \lim_{\beta \rightarrow \beta_c-}
   \Prob_\beta(\hwalk \cap \hat \walk^k) = 1 . 
\end{equation}
Also, every $\omega \in \hwalk \cap \hat \walk^k$ 
can be written uniquely as
\[  \omega = \omega^1 \otimes \omega^2, \;\;\;\;
             \omega^1 \in \ibridge_k, \;\;
  \omega^2 \in \hwalk^*. \]
This justifies the next definition.

\begin{definition}
If $k$ is a positive integer, let $Q_k$ denote
the probability measure on $\ibridge_k$ defined by 
\[    Q_k(\omega) = \beta_c^{|\omega|}, \;\;\;\;
   \omega \in \ibridge_k. \]
\end{definition}

We also write $Q_k$ for the extension to $\hwalk$ with 
$Q_k[\hwalk\setminus\ibridge_k]=0$. 
It may not be immediately obvious that this is a probability 
measure. Recall that $\ibridge_k$ is the set of bridges which 
are irreducible above height $k$. In other words, they are 
of the form $\omega^1 \bcat \omega^2 \bcat \cdots \bcat \omega^j$
for some $j$ with $\omega^i \in \ibridge$,
$h(\omega^1 \bcat \omega^2 \cdots \bcat \omega^{j-1})<k$,
and $h(\omega^1 \bcat \omega^2 \cdots \bcat \omega^j) \ge k$.
So we can think of the probability measure $Q_k$ as follows. 
We generate a sequence of i.i.d. irreducible bridges 
distributed according to the probability measure $Q$ and stop 
when the height of their concatenation is at least $k$. 
 
Note that we have defined two different but similar
probability measures $Q^k$ and $Q_k$.  $Q^k$ is
a measure on $\ibridge^k$, and $Q_k$ is a measure on
$\ibridge_k$.  The next proposition is immediate
from what we have done so far; it holds for
either of these measures.

\begin{proposition}  
\label{equiv_ensemble}
Suppose $\omega \in \hwalk$ with
$h(\omega) \leq j$. 
Then
\[   \Prob_{\beta_c}(\{\hat \omega  :\omega \prec \hat \omega\}):=
 \lim_{\beta \rightarrow \beta_c-}
    \Prob_{\beta}(\{\hat \omega  :\omega \prec \hat \omega\}) 
  = Q^j(\{\hat \omega  :\omega \prec \hat \omega\})
  = Q_j(\{\hat \omega:  \omega \prec \hat \omega \}).\]
\end{proposition}
\begin{proof}
We break the set 
$\{\hat \omega  :\omega \prec \hat \omega\}$ 
into two disjoint pieces $A$ and $B$.   
$A$ contains all the $\hat \omega$ for which there is an 
$\bar \omega \in \ibridge_j$ such that 
$\omega \prec \bar \omega \prec \hat \omega$.
$B$ contains all the $\hat \omega$ for which no such $\bar \omega$ 
exists. Then $A$ is the disjoint union
\[
A=\bigcup_{\bar \omega \in \ibridge_j : \omega \prec \bar \omega}
\hwalk(\bar \omega)
\]
The elements of $B$ are all extensions of $\omega$ which 
are irreducible above height $j$ and are not bridges.
$Q^j(\hwalk(\bar\omega))=\beta_c^{|\bar\omega|}=P_{\beta_c}(\hwalk(\bar\omega))$.  
$B$ has measure zero under $Q^j$ because all the elements
have less than $j$ irreducible bridges by construction. From 
\reff{shaneedit1} we see $\lim_{\beta\ra\beta_c}P_\beta(B)=0$.
The only elements with positive measure in $Q_j$ are those which are bridges
and are irreducible above height $j$. Each 
$\hwalk(\bar\omega)$ contains only one such SAW, namely $\bar \omega$.
So $Q_j(\hwalk(\bar\omega))=Q_j(\bar\omega)=\beta_c^{|\bar\omega|}$. 
\end{proof}

\subsection{Half-plane SAW starting at interior point}

We have considered half-plane SAW's that start
at a boundary point.  We can also consider walks
that start at an interior point.

\begin{definition}
If $z = x + iy \in \Z^2$ with $y > 0$, let $\hwalk^z$
denote the set of walks $\omega =[\omega_0,\ldots,\omega_n]$
with $\omega_0 = z$ and $\Im[\omega_j] >0$ for all $j$.
Let
\[    H(\beta,k) = \sum_{\omega \in \hwalk^{ki}}
     \beta^{|\omega|} = \sum_{\omega \in \hwalk^{x+ki}}
     \beta^{|\omega|}.  \]
\end{definition}

Note that if $\beta \geq 0$, then $H(\beta,k)$
is increasing in $k$;
Let $\Prob_{\beta,z}$ be the probability measure on
$\hwalk^z$ defined by
\[    \Prob_{\beta,x+ki}(\omega) =
           H(\beta,k)^{-1} \, \beta^{|\omega|}.\]
 In
the case $z=i$, there is an obvious bijection between
 $\hwalk^i$ 
and $\tilde \hwalk$. Hence 
\[   H(\beta,1) = \tilde H(\beta) \leq \beta^{-1}
  \, H(\beta) . \]

\begin{proposition}
For $\beta < \beta_c$ and each $k$, 
\begin{equation}  \label{mar20.4}
  H(\beta,k)  \leq (\beta^{-1}+ 3)^{k}
  \, H(\beta). 
\end{equation}
\end{proposition}

\begin{remark}  This estimate is not very
sharp but it will be useful.  The important fact
is that for each $k$, there is a $c_k 
= (2\beta_c^{-1} + 3)^k< \infty$
such that for all $\beta_c/2<
\beta < \beta_c$,
\begin{equation}  \label{mar19.9}
      H(\beta,k) \leq c_k \, H(\beta). 
\end{equation}
\end{remark}

\begin{proof}
We will proceed by induction; we have already
established \eqref{mar20.4}
for $k=1$. Recall that
$\beta < \beta_c < 1$.
  Suppose $k \geq 2$.
We will define an injection  
\[   \Phi: \hwalk^{ki} \longrightarrow
     \bigcup  \hwalk^z , \]
where the union is over $z \in \{(k-1)i, 1 + (k-1)i,
-1 + (k-1) i, (k-2)i\}.$ 
 We
partition $\hwalk^{ki}$ into three sets: $W_1$,
the 
walks that do not visit $(k-1)i$; $W_2$, the
 walks
whose first step is to $(k-1)i$; and $W_3$, the 
walks that visit $(k-1)i$ but not on the first step.

If  $ \omega \in W_1$, let
\[   \Phi(\omega) = [(k-1)i,ki] \acat
  \omega . \]
Then $|\Phi(\omega) | = |\omega| + 1$ and the first
step of $\Phi(\omega)$ is to $ki$. 
If $\omega = [ki,(k-1)i] \acat \tilde \omega
 \in W_2$,  let
\[   \Phi(\omega) = \tilde \omega. \]
Then $|\Phi(\omega)| = |\omega| - 1$ and $\Phi(\omega)$
does not visit $ki$.  In particular, the images
$\Phi(W_1)$ and $\Phi(W_2)$ are disjoint, and we can
conclude 
\[  \sum_{\omega \in W_1 \cup W_2}
      \beta^{|\omega|} \leq \beta^{-1}
 \, H(\beta,k-1).\]

If $\omega = [\omega_0,\ldots,\omega_n] \in
W_3$, let $j$ be the vertex with $\omega_j
 = (k-1)i$ and write
\[          \omega = \omega^1 \acat \omega^2, \]
where $\omega^1=[\omega_0,\ldots,\omega_j]$.
Note that  
\[ \omega_{j-1}\in \{(k-2)i, 1 + (k-1)i,
-1+(k-1)i\}.\]  
Define
\[ \Phi(\omega) =
\tilde \omega^1 \acat \omega^2 \mbox{ where }
  \tilde \omega^1 = [\omega_{j-1},\ldots,
\omega_0,\omega_j].\]
 Note that $|\Phi(\omega)| = |\omega|$
and that $\Phi$ maps $W_3$ injectively into
\[ \hwalk^{(k-2)i} \cup \hwalk^{1 + (k-1)i}
  \cup \hwalk^{-1+(k-1)i}.\]
 Hence
\[ \sum_{\omega \in W_3}
      \beta^{|\omega|}  \leq
   2 \, H(\beta,k-1) + H(\beta,k-2) \leq
   3 \, H(\beta,k-1) . \]
\end{proof}

\begin{lemma}
If $k < l$, then
\begin{equation}  \label{mar19.6}
 \sum_{\omega \in \hwalk^{ki}
  \setminus \hat \walk^l} \beta^{|\omega|}
  \leq (\beta^{-1} +3)^k \, 
\sum_{\omega \in \hwalk
  \setminus \hat \walk^l} \beta^{|\omega|}.
\end{equation}
In particular,
\[  \lim_{\beta \rightarrow \beta_c-}
    H(\beta,k)^{-1} \, \sum_{\omega \in \hwalk^{ki}
  \setminus \hat \walk^l} \beta^{|\omega|}
=0.\]
\end{lemma}

\begin{proof}  The proof of \eqref{mar19.6}
is  the same as the previous
proof noting that $\omega \in \hat \walk^l$
if and only if $\Phi(\omega) \in \hat \walk^l$.
Since $H(\beta,k)$ increases in $k$,
\[ H(\beta,k)^{-1} \, \sum_{\omega \in \hwalk^{ki}
  \setminus \hat \walk^l} \beta^{|\omega|}
  \leq \frac{(\beta^{-1} + 3)^k}
   {H(\beta)} \sum_{\omega \in \hwalk
  \setminus \hat \walk^l} \beta^{|\omega|}
   = (\beta^{-1} + 3)^k \, \Prob_\beta
         (\hwalk \setminus \hat \walk^l).\]
 The final assertion then follows from
\eqref{mar19.8}.
\end{proof}

\begin{definition}
If $\omega=[x+ ki,\omega_1,\ldots,
\omega_n] \in \hwalk^{x+ki}$ and $m  > k$,
we say that $\omega$ is an $m$-irreducible bridge
if there is an $l \geq m$ such that
 $l= \Im[\omega_n] = \max_{1 \leq j \leq n}
\Im[\omega_j]$ and $\omega$ has no 
cuts of level $\geq m$. (Note that an m-irreducible 
bridge is not necessarily a bridge in the sense 
described in the previous section.) Let
\[    \tilde I(\beta,k;m) = \sum \beta^{|\omega|}, \]
where the sum is over all $m$-irreducible bridges
in $\hwalk^{x+ki}$.
\end{definition}

\begin{proposition}  \label{mar20.prop2}
If $m >k$, then
\[ \sum_{\omega \in \hwalk^{ki} \cap \hat
   \walk^m} \beta^{|\omega|} =  \tilde I(\beta,k;m) 
 \, H(\beta) \]
In particular,
\[   \tilde I(\beta_c,k;m) =  \lim_{\beta \rightarrow
\beta_c-} \frac{H(\beta,k)}{H(\beta)} < \infty, \]
so $\tilde I(\beta_c,k,m)$ is independent of $m$, and we denote it just by
 $\tilde I(\beta_c,k)$.
\end{proposition}

\begin{proof}
The first assertion is immediate using the unique
decomposition of $\omega \in \hwalk^{ki} \cap \hat
   \walk^m$ as $\omega = \omega^1 \otimes \omega^2$
where $\omega^1$ is an $m$-irreducible bridge
and $\omega^2 \in \hwalk$.   
From the previous lemma, we see that 
\[  H(\beta,k) \sim \sum_{\omega \in \hwalk^{ki} \cap \hat
   \walk^m} \beta^{|\omega|} , \;\;\;\;
  \beta \rightarrow \beta_c-. \]
Using \reff{mar19.9}, the  second assertion follows.
\end{proof}

This proposition justifies the following definitions.

\begin{definitions}$\;$

\begin{itemize}

\item 
For any $z = x+ik \in \Z^2$ with $k >0$ and $m > k$, let
 $\tilde Q_{z,m}$ denote the probability measure
on $m$-irreducible bridges
in $\hat \hwalk^z$ given by
\[     \tilde Q_{z,m}(\omega) = \frac{\beta
   _c^{|\omega|}}  {\tilde I(\beta_c,k)}.\]

\item
The infinite half-plane SAW starting at
$z$ is obtained by choosing
\[   \omega^1 \bcat \omega^2 \bcat \cdots , \]
where $\omega^1,\omega^2,\ldots$ are independent;
$\omega^1$ is chosen from $\tilde Q_{z,m}$,
and $\omega^2,\omega^3,$ $\ldots$ are chosen from
$Q$. (The fact that this definition does not depend on $m$ 
follows easily from the fact that $\tilde I(\beta_c,k,m)$ is 
independent of $m$.)

\end{itemize}

\end{definitions}

\begin{proposition}  Suppose $z = x+ik \in \Z^2$
with $k >0$ and $m \geq k$.  For each $m$-irreducible
bridge $\omega \in
\hwalk^z$, let $\hwalk(\omega)$ denote the set of
walks $\tilde \omega \in \hwalk^z$ of the form
\[       \tilde \omega = \omega \otimes \omega^1,
\;\;\;\; \omega^1 \in\hwalk. \]
Then
\[  \lim_{\beta \rightarrow \beta_c-}
  \Prob_{\beta,z}[\hwalk(\omega)] = \tilde Q_{z,m}
 (\omega) = \frac{\beta_c^{|\omega|}}
   {\tilde I(\beta_c,k)} .\]
\end{proposition}

\begin{proof}  This is proved the same way as 
proposition \ref{equiv_ensemble}.
\end{proof}

\begin{definition}
Suppose $V \subset \Z^2$ 
and $z = x + yi \in \Z^2$ with $y >0$.
Then   
$\hwalk^{z,V}$ denotes the set of walks
$\omega=[z,\omega_1,\ldots,\omega_n] \in
\hwalk^z$  such that $\omega \cap V =
 \emptyset $. In particular, $\hwalk^z
= \hwalk^{z,\emptyset}.$
\end{definition}

\begin{proposition}  \label{mar20.prop3}
Suppose $z = x + li \in \Z^2$ with 
$0 < l \leq k$, $V \subset \{x+iy \in \Z^2:
 y \leq k\}$, and $m > k$.  Then
\[      \Prob_{\beta_c,z}[\hwalk^{z,V}]
    = \sum \tilde Q_{z,m}(\omega) ,\]
\[   \lim_{\beta \rightarrow \beta_c-}
     \sum \beta^{|\omega|}
    =   \Prob_{\beta_c,z}[\hwalk^{z,V}] \, \tilde
   I(\beta_c;l), \]
where the sums are over all $m$-irreducible bridges
$\omega \in \hwalk^z$ with $\omega \cap V = \emptyset$.
\end{proposition}

\subsection{Limit of counting measure}  \label{unifsec}

Another natural measure on $\hwalk$ is $\Prob^n$,
the uniform measure on all walks in $\hwalk$
of length $n$.  
(We think of $\Prob^n$ as a probability measure on $\hwalk$ by
defining $\Prob(\omega)=0$ if $|\omega| \neq n$.)
This is the measure that is easiest
to simulate.  Here we show that the limit as $n \rightarrow
\infty$ of $\Prob^n$ is the infinite half-plane SAW.
 Let $Y_n$ denote
the cardinality of $\{\omega \in \hwalk: |\omega|
 = n\}.$  We will use without proof Kesten's result \cite{kestena}
\begin{equation}  \label{kesten}
          \lim_{n \rightarrow \infty}
    \frac{Y_{n}}{Y_{n+2}} = \beta_c^2.
\end{equation}
(Kesten proves this for walks in the full plane. His argument works for 
half-plane walks as well.) 
Suppose $\omega^1,\ldots,\omega^j \in \ibridge$
with $m = |\omega^1| + \cdots + |\omega^j|$,
and let
$\hwalk(\omega^1,\ldots,\omega^j)$ be as
in Proposition \ref{mar19.prop1}. Then
\begin{equation}  \label{mar20.1}
   \Prob^n[\hwalk(\omega^1,\ldots,\omega^j)]
=      \frac{ \, Y_{n-m}}{Y_n} , 
\end{equation}
where $Y_{k} = 0 $ if $k < 0$.

\begin{proposition}  \label{uniform_conv}
If $\omega^1,\ldots,\omega^j
\in \ibridge$ with $|\omega^1| + \cdots +
|\omega^j| = m$, then
\[    \lim_{n \rightarrow \infty}
     \Prob^n [\hwalk(\omega^1,\ldots,\omega^j)]
   = \beta_c^m . \]
\end{proposition}

\begin{remark}
 As before, let $\hwalk^\infty(\omega^1,\ldots,\omega^j)$ denote 
the set of infinite length walks that start with 
$\omega^1 \bcat \cdots \bcat \omega^j$. Then 
$Q^\infty[\hwalk^\infty(\omega^1,\ldots,\omega^j)]=\beta_c^m$. So the 
proposition says that 
\[    \lim_{n \rightarrow \infty}
     \Prob^n [\hwalk(\omega^1,\ldots,\omega^j)]
   = Q^\infty[\hwalk^\infty(\omega^1,\ldots,\omega^j)]
\]
This is the precise sense in which the limit as $n \rightarrow \infty$
of $\Prob^n$ is the infinite half-plane SAW, i.e., $Q^\infty$.
\end{remark}

\begin{proof}   
In \cite{LSW_SAW} it was shown that 
\begin{equation}  
          \lim_{n \rightarrow \infty}
    \frac{Y_{n}}{Y_{n+1}} = \beta_c.
\end{equation}
If we assume this result, then the proof of the proposition is immediate. 
Instead we give a proof that uses only Kesten's result \eqref{kesten}.

Since for fixed $j$, 
\[  \sum_{\eta^1,\ldots,\eta^j \in \ibridge}
  \beta_c^{|\eta^1| + \cdots + |\eta^j|} = 1, \]
it suffices to show for every $\omega^1,\ldots,\omega^j\in\ibridge $
\[    \liminf_{n \rightarrow \infty}
     \Prob^n [\hwalk(\omega^1,\ldots,\omega^j)]
    \geq \beta_c^m . \]

We will assume that $m$ is even.
The case $m$ odd is done the same way. From
\eqref{kesten} and \eqref{mar20.1}, we see that
\[   \lim_{n \rightarrow \infty}
   \Prob^{2n}[\hwalk(\omega^1,\ldots,\omega^j)]
   = \beta_c^m . \]
We need to prove also that
\begin{equation}  \label{mar20.3}
\liminf_{n \rightarrow \infty}
   \Prob^{2n+1}[\hwalk(\omega^1,\ldots,\omega^j)]
   \geq  \beta_c^m . 
\end{equation}
 Suppose
$\eta^1,\ldots,\eta^k \in \ibridge$ with
$|\eta^1|,\ldots,|\eta^{k-1}|$ even and $
l = |\eta^1| + \cdots + |\eta^k|$
odd. Then, by the same reasoning
\begin{equation}  \label{mar20.2}
 \lim_{n \rightarrow \infty}
 \Prob^{2n+1}[\hwalk(\omega^1,\ldots,\omega^j,
\eta^1,\ldots,\eta^k)]
   = \beta_c^{m + l}. 
\end{equation}
Note that
\begin{equation}\label{shaneedit2}  
{\sum}^\dagger \, \beta_c^{|\eta^1| + \cdots + |\eta^k|} = 1, \end{equation}
where the sum ${\sum}^\dagger$ is over all $k$ and all $\eta^1,\ldots,
\eta^k \in \ibridge$ 
such that $|\eta^1|, \ldots, |\eta^{k-1}|$ are
even and $|\eta^k|$ is odd.
 (Indeed, this is what one gets if
one selects $\eta \in \ibridge$  from the distribution
$Q$ until one gets a bridge $\eta$
 of odd length.  Since the
probability of getting an $\eta$ of odd length at
each step is $\rho > 0$ one gets one with probability one.)
Similarly,
\[ \Prob^{2n+1}[\hwalk(\omega^1,\ldots,\omega^j)]
   \geq {\sum}^\dagger \,
\Prob^{2n+1}[\hwalk(\omega^1,\ldots,\omega^j,
\eta^1,\ldots,\eta^k)].\]
where the sum is over $k$ and $\eta^1,...,\eta^k$ subject to the 
same constraints as before.
Using this, \eqref{mar20.2} and \eqref{shaneedit2} we get \eqref{mar20.3}.
(The reader may observe that 
the interchange of limits only allows a
statement about the liminf rather than a statement about
the limit, but we have noted that this suffices.)
\end{proof}

\begin{definition}  If $A \subset \walk$ is any set
of SAW's, define
\[          M(A) = \lim_{n \rightarrow \infty}
       \frac{\#\{\omega \in A: |\omega| = n \}}{Y_n}, \]
assuming the limit exists.
\end{definition}

Using Proposition \ref{mar20.prop2} and a proof as in the previous
proposition, we can show that
\[      M(\hwalk^{ik}) = M(\hwalk^{x+ik}) = \tilde I(\beta_c,k) .\]
Similarly, as in Proposition \ref{mar20.prop3}, we have
\[     M(\hwalk^{ik,V}) = \tilde I(\beta_c,k) \, \Prob_{\beta_c,ik}
    [\hwalk^{ik,V}]. \]

\subsection{SAW in a strip}
\label{sect_saw_strip}

In the previous section we considered the self-avoiding walk in 
the half plane. For every such walk the self-avoidance implies that 
$\lim_{n \rightarrow \infty} |\omega_n| =\infty$. 
So the probability measure on the 
infinite walk is supported on walks between $0$ and $\infty$,  
two boundary points of the upper half plane.  
In this section we consider the SAW in a strip starting 
at one boundary point and ending either at a fixed boundary 
point on the other side of the strip or at any boundary 
point on the other side of the strip.

We fix a positive integer $y$ and consider the strip 
\[
 \strip_y = \{z :  0 < \Im(z) < y\} 
\]
Fix a point $x+iy$ on the upper boundary of the strip. The 
SAW from $0$ to $x+iy$ in the strip is the probability measure 
on SAW's that start at $0$, end at $x+iy$ and in between stay in the strip 
that is defined as follows. 

\begin{definition} 
The probability measure
$\Q_x^y$ on the SAW's in the strip $\strip_y$ from $0$ to $x+iy$ 
is defined by 
\[
\Q_x^y(\omega) = \frac{1}{Z_{\strip_y}(0,x+iy)} \beta_c^{|\omega|} 
\]
where $Z_{\strip_y}(0, x+iy)$ is defined by the requirement that this be 
a probability measure. 
\end{definition} 

This definition only makes sense if $Z_{\strip_y}(0,x +iy)$ is finite. 
Note that if we take one of the SAW's we are considering and 
ignore the last step, then we have a bridge with height $y-1$. 
It need not be irreducible, but it can be decomposed into irreducible 
bridges. Let $n$ be the number of irreducible bridges. 
The sum of $\beta_c^{|\omega|}$ over the $\omega$ in $Z_{\strip_y}(0,x +iy)$ that have 
exactly $n$ irreducible bridges is bounded by $1$ by Kesten's relation
(Prop. \ref{kesten_prob}).
Since $n$ is at most $y-1$, it follows that $Z_{\strip_y}(0,x+iy)$ is finite.

We can also consider SAW's in the strip that start at $0$ and end 
at any point $x+iy$ on the upper boundary. We define a probability 
measure $\Q^y$ on this set of walks in the obvious way.
\begin{definition} 
$\Q^y$ is the probability measure on the SAW's in the strip $\strip_y$ 
from $0$ to some point $z$ with $\Im[z]=y$ given by 
\[
\Q^y(\omega) = \frac{1}{Z_{\strip_y}(0,\integers +iy)} \beta_c^{|\omega|} 
\]
where $Z_{\strip_y}(0,\integers +iy)$ is defined by the requirement 
that this be a probability measure. 
\end{definition} 
The same argument that shows $Z_{\strip_y}(0,x +iy)$ is finite shows 
$Z_{\strip_y}(0,\integers +iy)$ is finite. 
Note that we use a superscript on $\Q$ for probability measures on 
walks that end anywhere on the upper boundary of the strip and both a
subscript and a superscript on $\Q$ for probability measures that 
end at a prescribed point on the 
boundary. Obviously, $\Q_x^y$ can be obtained from $\Q^y$ by conditioning 
on the event that the walk ends at $x+iy$. 

Recall that an infinite walk $\omega$ in the half plane has a cut at 
level $y$ if there is a bridge $\omega_1$ with 
height $y$ and an infinite half plane walk $\omega_2$ such that 
$\omega=\omega_1 \bcat \omega_2$.
(Note that this implies that the first bond in $\omega_2$ is the 
only bond in the walk that goes between height $y$ and $y+1$.)
The key to our simulations is the following observation. 

\begin{proposition} \label{half_to_strip}
Let $y$ be a positive integer. Let $Q^\infty$ be the probability measure on 
infinite SAW's in the half plane from the previous sections. If we condition
$Q^\infty$ on the event that the walk has a cut at level $y-1$ 
and only consider the walk up to height $y$, 
then the resulting probability measure is $\Q^y$.
\end{proposition}

\begin{proof}
Let $E_{y-1}$ be the event that $\omega$ has a cut at level $y-1$. 
Recall that for irreducible bridges $\omega^1,\ldots,\omega^j$, 
$\hwalk^\infty(\omega^1,\ldots,\omega^j)$ is the set of infinite walks that 
begin with $\omega^1 \bcat \cdots \bcat \omega^j$. 
We can write $E_{y-1}$ as a disjoint union of these events:
\begin{equation}
E_{y-1} = \bigcup_{j=1}^{y-1}\;\; \bigcup_{\omega^1,\ldots,\omega^j \in \ibridge: 
h(\omega^1) + \cdots + h(\omega^j)=y-1} \,
\hwalk^\infty(\omega^1,\ldots,\omega^j)
\end{equation}
By the definition of $Q^\infty$, 
\[
Q^\infty[\hwalk^\infty(\omega^1,\ldots,\omega^j)]=
\beta_c^{|\omega^1|+\cdots+|\omega^j|}
\]
So when we condition on $E_{y-1}$, the probability of 
$\hwalk^\infty(\omega^1,\ldots,\omega^j)$ is proportional to 
$\beta_c^{|\omega^1|+\cdots+|\omega^j|}$.
If we only consider the walk up to height $y$, then each 
$\hwalk^\infty(\omega^1,\ldots,\omega^j)$ corresponds to a single walk, namely, 
$\omega^1 \bcat \cdots \bcat \omega^j$, concatenated with a vertical 
bond from height $y-1$ to $y$. 
So under the conditioning given in the proposition, 
the probability of a walk in the strip is proportional to 
$\beta_c^{|\omega^1|+\cdots+|\omega^j|}$.
This is what its probability should be under $\Q^y$, thus 
proving the proposition. 
\end{proof}

\section{Simulations}
\label{sect_sim}

\subsection{SLE partition functions}
\label{sect_sim_part_func}

The pivot algorithm provides a fast Markov chain Monte Carlo algorithm
for simulating the SAW in the full plane or the half plane. 
(For an introduction to this algorithm see \cite{Madras}. For the 
fastest implementation of the algorithm to date see \cite{clisby}.)
However, the pivot algorithm cannot be used for the SAW in 
most simply connected domains. 
The algorithms that do work in these cases are much slower 
than the pivot algorithm. Proposition  \ref{mar19.prop1} provides 
the key to a fast simulation of the SAW in a strip of height $y$.
We simulate the SAW in the half plane and condition on the event 
that the walk has a cut at level $y-1$. The probability of this 
event is small, but still large enough to allow us to generate
large numbers of samples. Proposition \ref{half_to_strip} says that 
the distribution of the portion of the walk in the strip of 
height $y$ is $\Q^y$. 

Of course, the scaling limit of $\Q^y$ is not chordal SLE. 
We could condition further on the event that the 
walk ends at a particular point $x+iy$ to get
$\Q_x^y$, whose scaling limit should be chordal SLE. 
However, this would be conditioning on an event with 
very small probability. Instead we take the following approach. Let 
$\rho(x)$ be the scaling limit of the 
density of the unique point $x+iy$ on the walk at height $y$
under the measure $\Q^y$. If we integrate chordal SLE$_{8/3}$
between $0$ and $x+i$ in the strip of height 1
against this density, then we should 
obtain the scaling limit of $\Q^y$. So if we can compute
$\rho(x)$, then we can use known results about chordal SLE to 
make predictions for the scaling limit of $\Q^y$.
We will derive a conjecture for $\rho(x)$ using SLE partition functions.

For a simply connected domain $D$ and points $z,w$ on its boundary, we define
the SLE partition function $H_D(z,w)$ by requiring that it 
be conformally covariant in the following sense. If $\Phi$ is 
a conformal map on $D$, then 
\begin{equation}
H_D (z,w)  = 
|\Phi^\prime(z) \Phi^\prime(w)|^{5/8} 
H_{\Phi(D)}(\Phi(z),\Phi(w)) 
\label{conf_cov}
\end{equation}
This defines $H_D(z,w)$ up to specifying its value for one particular 
choice of $D,z,w$. We follow the usual convention of taking 
$H_\Half(0,1)=1$. 
This partition function is believed to be related to the 
total mass of the scaling limit of SAW's in $D$ from $z$ to $w$. 
This statement must be interpreted with caution as there are 
lattice effects that persist in the scaling limit.
We will motivate this conjecture by arguing that it is true in 
two special cases. 

In the first special case, $D$ will be $\Half$ and $\Phi$ will just 
be a dilation.
Let 
\[
Z_\Half(0,n) = \sum_{\omega:0 \rightarrow n, \omega \subset \Half} \beta_c^{|\omega|} 
\nonumber
\]
The sum is over all self-avoiding walks that start at $0$, end 
at $n$  and stay in the upper half plane except for their endpoints.
We assume that there is an exponent $\scaleexp$ such that 
the limit 
\begin{equation}
c=\lim_{n \rightarrow \infty} Z_\Half(0,n) n^\rho
\label{saw_weight}
\end{equation}
exists. There is not a proof that $Z_\Half(0,n)$ is finite.  
Now let $x >0$. Let $[nx]$ denote the integer closest to $x$. 
We have
\[
\lim_{n \rightarrow \infty} Z_\Half(0,[nx]) n^\scaleexp = 
\lim_{n \rightarrow \infty}  Z_\Half(0,[nx]) [nx]^\scaleexp 
\frac{n^\scaleexp}{[nx]^\scaleexp}
= \frac{c}{x^\scaleexp}
\nonumber
\]
This shows that 
\[
\lim_{n \rightarrow \infty} \frac{ Z_\Half(0,[nx])}{Z_\Half(0,n)} = 
\frac{1}{x^\rho}
\label{scaling}
\]
If we use  \eqref{conf_cov} with $\Phi$ equal to a dilation, then 
we have 
\begin{equation}
\frac{H_\Half(0,x)}{H_\Half(0,1)} = H_\Half(0,x)= \frac{1}{x^{5/4}}.
\label{dilation_case}
\end{equation}
So if \eqref{saw_weight} is true with $\rho=5/4$, then 
$H_{\Half}(0,x)$ does indeed give the relative weight of the scaling 
limit of the SAW from $0$ to $x$ in $\Half$.

For the argument for the second special case of \eqref{conf_cov} we use the 
following result of Lawler, Schramm and Werner \cite{LSW_restriction}.
\begin{theorem}\label{LSWthm1}
Suppose $A\subset\Halfbar$ is compact and $\Half\setminus A$ is 
simply connected with $0\notin A$.  
Let $\Prob_{\Half,0,\infty}$ denote the probability measure on simple
curves in $\Half$ from $0$ to $\infty$ given by chordal SLE$_{8/3}$. 
If $\Phi_A:\Half\setminus A \rightarrow \Half$ denotes the unique 
conformal map that fixes $0$ and $\infty$ and has $\Phi_A'(\infty)=1$, 
then 
\begin{equation}\label{LSWeq1}
\Prob_{\Half,0,\infty}[\gamma \cap A = \emptyset]=\Phi_A'(0)^{5/8}
\end{equation}
\end{theorem}

Let $D^\prime$ be a simply connected domain and $z,w$ two points on 
its boundary. Let $D$ be a simply connected domain which is a
subset of $D^\prime$ and such that $z,w$ also belong to the boundary of $D$. 
Let $\Phi$ be a conformal map of $D$ to $D^\prime$ which fixes $z$ and $w$. 
Then the definition of the SLE partition function says
\begin{equation}
H_D (z,w)  = 
|\Phi^\prime(z) \Phi^\prime(w)|^{5/8} H_{D^\prime}(z,w) 
\end{equation}
Now consider the total mass of the SAW in $D$ from $z$ to $w$ divided 
by the total mass of the SAW in $D^\prime$ from $z$ to $w$. This ratio
gives the probability that a SAW in $D^\prime$ from $z$ to $w$ remains in 
the smaller domain $D$. This should be given by the probability
that a chordal SLE$_{8/3}$ in $D^\prime$ from $z$ to $w$ remains in 
$D$, i.e., 
by $\Prob_{D^\prime,z,w}(\gamma \cap (D^\prime \setminus D) = \emptyset)$, 
where $\Prob_{D^\prime,z,w}$ denotes this chordal SLE measure.
We can compute this probability using Theorem \ref{LSWthm1}. 
We will show that it is given by $|\Phi^\prime(z) \Phi^\prime(w)|^{5/8}$, 
i.e., by the ratio of the SLE partition functions 
$H_D(z,w) / H_{D^\prime}(z,w) $.

Let $\psi(z)$ be a conformal map of $D^\prime$ to $\Half$ with $\psi(z)=0$ and 
$\psi(w)=\infty$. 
By the conformal invariance of SLE, 
\begin{equation}
\Prob_{D^\prime,z,w}(\gamma \cap (D^\prime \setminus D) = \emptyset) = 
\Prob_{\Half,0,\infty}(\gamma \cap \psi(D^\prime \setminus D) = \emptyset) .
\end{equation}
Let $\phi = \psi \circ \Phi \circ \psi^{-1}$.  Then $\phi$ maps  
$\Half \setminus \psi(D^\prime \setminus D)$ to $\Half$ and fixes $0$ and 
$\infty$. So it satisfies the hypotheses of Theorem \ref{LSWthm1} 
except that its derivative at $\infty$ need not be $1$. 
An easy computation shows 
$\phi(z) \sim z/\Phi^\prime(w)$ as $z \rightarrow \infty$. 
So $z \mapsto \Phi^\prime(w) \phi(z)$ satisfies the hypotheses. 
The chain rule shows $\phi^\prime(0)=\Phi^\prime(z)$, so 
\begin{equation}
\Prob_{\Half,0,\infty}(\gamma \cap \psi(D^\prime \setminus D) = \emptyset) 
= |\Phi^\prime(z) \Phi^\prime(w)|^{5/8} 
\end{equation}
Thus we have derived \eqref{conf_cov} for the special case that
$\Phi(D)$ is a superset of $D$ with $z,w$ on its boundary too, 
and $\Phi$ fixes $z$ and $w$. 

\subsection{Boundary density for the strip}

Our simulations are all for the SAW in a strip. 
We use $\strip$ to denote the strip of height $1$, i.e., 
$\{z\in\Half: 0 < \Im(z) < 1\}$.  
We use SLE partition functions to compute the conjecture for 
the boundary density for the scaling limit of the SAW in $\strip$
starting at $0$ and ending somewhere on the upper boundary of the strip.
So we need to compute $H_\strip(0,x+i)$. 
Let $f(z)=e^{\pi z}-1$. This sends the strip to the upper half plane, 
sending $0$ to $0$ and $x+i$ to $-e^{\pi x}-1$. 
We have $f^\prime(0)=\pi $, $f^\prime(x+i)= -\pi e^{\pi x}$. 
So using \eqref{conf_cov} and the special case \eqref{dilation_case}, 
we have
\begin{eqnarray*}
H_\strip(0,x+i) &=& (\pi^2 e^{\pi x})^{5/8} \, H_\Half(0,-e^{\pi x} -1) \\
&=& \left[ \frac{ \pi^2 e^{\pi x}}{(e^{\pi x} + 1)^2} \right]^{5/8}
= \left[ \frac{\pi^2}{4 \cosh^2(\pi x/2)} \right]^{5/8}
\end{eqnarray*}
Thus our conjecture for the density for the SAW in the strip from $0$ 
to the point $x+i$ on the upper boundary is 
\begin{equation}\label{RealDensityEq}
\rho(x) = c \left[\cosh\left(\frac{\pi}{2} x \right)\right]^{-5/4}, \quad 
-\infty < x < \infty
\end{equation} 
where the constant $c$ normalizes the density.

Our simulations generated SAW's in the half plane with $N=10,000$ steps.
We condition on the event that the walk has a cut at a fixed level $h$ 
which we take to be $h=0.2 N^{3/4}=200$. 
In practice this means that we run the pivot algorithm to generate a 
sequence of SAW's. These walks are highly correlated, so we only 
look at the walk every $100,000$ steps of the Markov chain to see if 
the walk has a cut at the prescribed level. If it does, we use it 
to compute a sample of the random variables we are studying. 
We generated a total of $496,000$ samples. 
For the particular $N$ and $h$ that we used, approximately 
$27 \%$ of the SAW's have a cut at the prescribed height. 
We expect this fraction to decrease to zero as $h$ goes to infinity. 

Figure \ref{exitdistx} shows our test of conjecture 
\eqref{RealDensityEq} for the boundary density. 
We plot the cumulative distribution for this density 
for a simulation of the SAW and for the conjectured density.
The two curves are indistinguishable on the scale of the figure. 
The left inset shows a blow-up of a section of the two curves 
to show the size of the difference. The right inset is a plot of the difference
of the theoretical and empirical distribution. This difference is 
typically on the order of $1/10$ of a percent. The very jagged nature
of this curve is a reflection of the effect of the nonzero lattice 
spacing. On a small scale the simulation is actually simulating a 
discrete random variable. 

\begin{figure}[!h]
\includegraphics{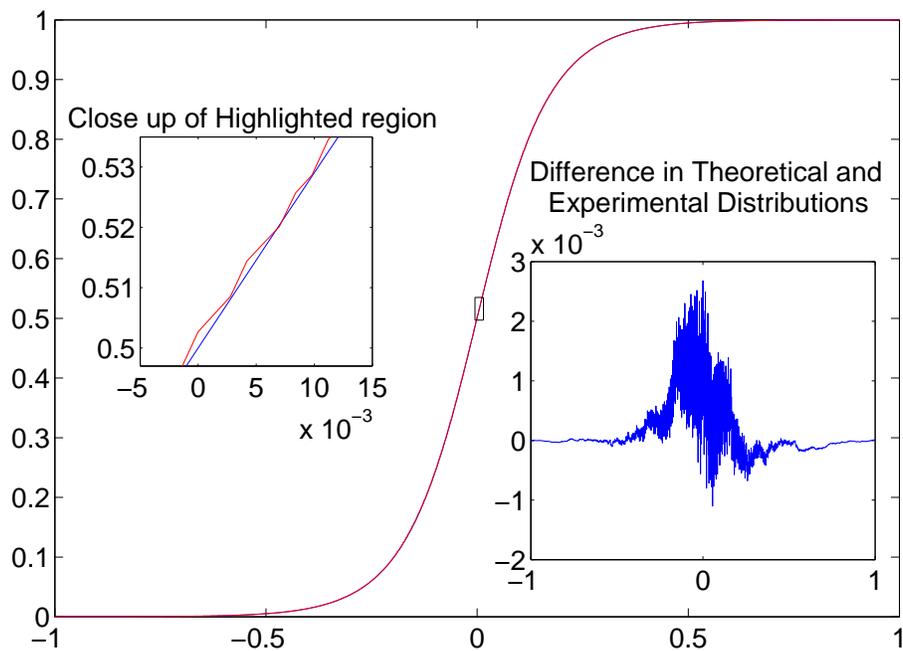}
\caption{\leftskip=25 pt \rightskip= 25 pt 
Comparison of simulations of boundary density for the SAW in a strip and 
the conjectured density using SLE partition functions.
}
\label{exitdistx}
\end{figure}

\subsection{Passing left of a point}

Now that we have found $\rho(x)$, we can use it together with
results for chordal SLE$_{8/3}$ to make predictions for the 
scaling limit of $\Q^y$. 
The probability of certain random events depending on the geometry 
of the SLE$_{8/3}$ curve can be explicitly computed.  
We have already seen one such formula in Theorem
\ref{LSWthm1}.

The second formula for chordal SLE that we will use is Schramm\rq s 
{\it left-passing probability} of a point $z\in \Half$ with 
respect to the SLE$_\kappa$ generating curve \cite{Schramm2}.  
(Schramm called this 
a \lq left crossing probability,\rq \ 
but we use slightly different terminology here.)  
Schramm's formula applies to any $\kappa\in(0,8)$.
The definition of left-passing is given in terms of winding numbers.  
For $\kappa=8/3$, the generating curve is a simple curve;  
an equivalent definition of left-passing is 
\begin{equation}
\Prob_{\Half,0,\infty}[\gamma \text{ passes left of } z]
=\Prob_{\Half,0,\infty}[z\in H_\infty^+]
\nonumber
\end{equation}  
where $H_\infty^+$ is defined to be the connected component of 
$\Halfbar \setminus\gamma[0,\infty)$ that contains 
$\reals^+:=\{x\in\reals:x>0\}$. Since $\gamma$ is simple and 
$\gamma(t) \rightarrow \infty $ w.p.1, $\Halfbar\setminus\gamma[0,\infty)$ 
has exactly two connected components.
For $\kappa = 8/3$, the formula reduces to 
\begin{equation}\label{83CrossProb}
\Prob_{\Half,0,\infty}[\gamma \text{ passes left of } z] 
= \frac{1}{2}[1 + \cos(\arg(z))]. 
\end{equation}

For the SAW in the half plane, simulations were compared with the 
left-crossing formula and with several applications of 
Theorem \ref{LSWthm1} in \cite{Kennedya,Kennedyb}.
Excellent agreement was found. We emphasize that these past Monte 
Carlo tests of SLE predictions for the SAW have all been for the 
SAW in the half plane or in a slit full plane. The simulations in 
this paper are the first tests of SLE predictions for the SAW in a strip.
More significantly, they are the first tests of the SLE partition function 
prediction \eqref{RealDensityEq}.

Recall that $\strip=\{z\in\Half: 0 < \Im (z) < 1\}$.  
The  map
\begin{equation}\label{StripToHalfPlaneMap}
 \Psi_\xi(z) = \frac{e^{\pi z}-1}{e^{\pi z}+e^{\pi \xi}}
\end{equation}
maps $\strip$ to $\Half$ with 
$\xi+i\mapsto \infty$ and $0\mapsto 0$.  
We use this map to transform probabilities involving chordal SLE 
in the strip into chordal SLE in the half plane. 

If $z$ is a point in the interior of the strip $\strip$,
then 
\begin{eqnarray*}
\Prob_{\strip,0,x+i}(\gamma \text{ passes left of } z) 
 &=& \Prob_{\Half,0,\infty}(\gamma \text{ passes left of } \Psi_x(z))  \\
 &=&  \frac{1}{2}[1 + \cos(\arg(\Psi_x(z)))]. 
\end{eqnarray*}
So we expect that 
\[
\lim_{y \ra \infty} \Q^y \left( \frac{\omega}{y}  \text{ passes left of } z \right) 
= \int_{-\infty}^\infty  \frac{1}{2}[1 + \cos(\arg(\Psi_x(z)))] \, \rho(x) \, dx
\]

In our simulations we sample the event that the walk passes left of $z$ 
for a grid of values of $z$. 
The grid is $400$ points wide horizontally and $100$ points wide vertically.
We then use the probabilities of these events to 
compute a contour plot for the probability of the curve passing left as 
a function of $z$. 
Figure \ref{passright} shows the comparison of the theoretical probabilities
and the simulations. Contour plots were generated for the contours for 
probabilities $0.1,0.2,0.3, \dots,0.8,0.9$. 
The solid lines are the theoretical contours and the circle are points on 
the empirical contours at $y$ values of $0.05,0.10,0.15, \dots, 0.95$. 

\begin{figure}[!h]
\includegraphics{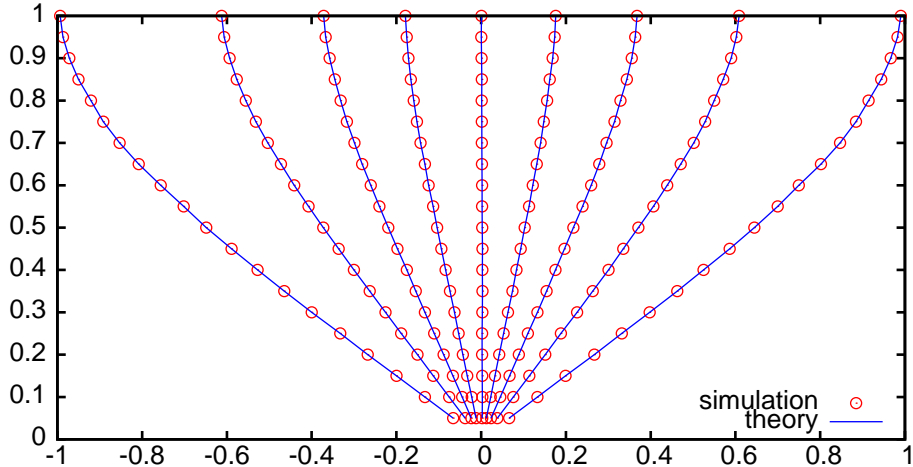}
\caption{\leftskip=25 pt \rightskip= 25 pt 
Comparison of simulations of the probability of passing left
for the SAW in a strip and the conjectured probability function using SLE. 
The plot is a contour plot in the strip which 
shows the curves where the probability is 0.1,0.2,...,0.9.
}
\label{passright}
\end{figure}

\subsection{Right-most excursion}

Our next test considers the right-most point on the SAW in the strip, i.e., 
the random variable 
\[
\max_{0\leq j\leq |\omega|} \Re \, \omega_j
\]
We conjecture that in the scaling limit its distribution is given by 
\[ \label{rightmostconj}
\lim_{y \ra \infty} \Q^y \left( 
\max_j \frac{\Re \, \omega_j}{y} 
\le x \right) = \int_{-\infty}^\infty \Prob_{\strip,0,\xi+i} 
(\max_t \Re \, \gamma(t) \leq x) \, \rho(\xi) \, d\xi
\]
where $\Prob_{\strip,0,\xi+i}$ denotes the probability measure for chordal 
SLE$_{8/3}$ in the strip $\strip$ from $0$ to $x+i$. 

To compute $\Prob_{\strip,0,\xi+i} (\max_t \Re \, \gamma(t) \leq x)$, we 
first note that this probability is zero if $\xi > x$. 
For $\xi < x$, we observe that the random variable is $<x$ if and only if 
the curve does not hit the portion of the strip given by $\Re \, z \ge x$. 
So for $\xi < x$ we have 
\begin{eqnarray*}
 \Prob_{\strip,0,\xi+i} \left(\max_t \Re \, \gamma(t) < x \right)
  &=& \Prob_{\strip,0,\xi+i} 
  \left(\gamma \cap \{z \in \strip: \Re \, z \geq  x \} 
  = \emptyset \right) \\
  &=& \Prob_{\Half,0,\infty} (\gamma
   \cap \Psi_\xi(\{z\in \strip: \Re \, z \geq x\}) = \emptyset)\\
  &=& \Prob_{\Half,0,\infty} 
    (\gamma \cap \{z\in\Half: |z-a(x;\xi)|\leq c(x;\xi)\} = \emptyset),
\end{eqnarray*}
where $a(x;\xi)$ and $c(x;\xi)$ are given by,
\begin{eqnarray*}
a(x;\xi)&=&
   \frac{1}{2}\left(\frac{\exp{\pi x}-1}{\exp{\pi x}+\exp{\pi \xi} }
  +\frac{\exp{\pi x}+1} {\exp{\pi x}-\exp{\pi \xi} }\right) 
\text{   and } \\ 
c(x;\xi)&=&
   \frac{1}{2}\left| \frac{\exp{\pi x}-1}{\exp{\pi x}+\exp{\pi \xi} }
  -\frac{\exp{\pi x}+1}{\exp{\pi x}-\exp{\pi \xi} }\right| . 
\end{eqnarray*}
We can now appeal to equation \eqref{LSWeq1}.  The map 
$\Phi_A$ in \eqref{LSWeq1} is given by
\begin{equation*}
\Phi_A(z)=\frac{(c-a)^2}{a}\left(\frac{\left(\frac{c-(z-a)}{c+(z-a)}\right)^2
  -\left(\frac{c+a}{c-a}\right)^2}
  {\left(\frac{c-(z-a)}{c+(z-a)}\right)^2-1}\right).
\end{equation*}
where $a=a(x;\xi), c=c(x;\xi)$. 

Evaluating the derivative at $z=0$ gives
\[
\lim_{y \ra \infty} \Q^y \left( 
\max_j \frac{\Re \, \omega_j}{y} 
\le x \right) = \int_{-\infty}^x
\left[1-\left(\frac{c(x;\xi)}{a(x;\xi)}\right)^2\right]^{5/8} 
\rho(\xi) \, d\xi.
\]
\begin{figure}[tbh]
\includegraphics{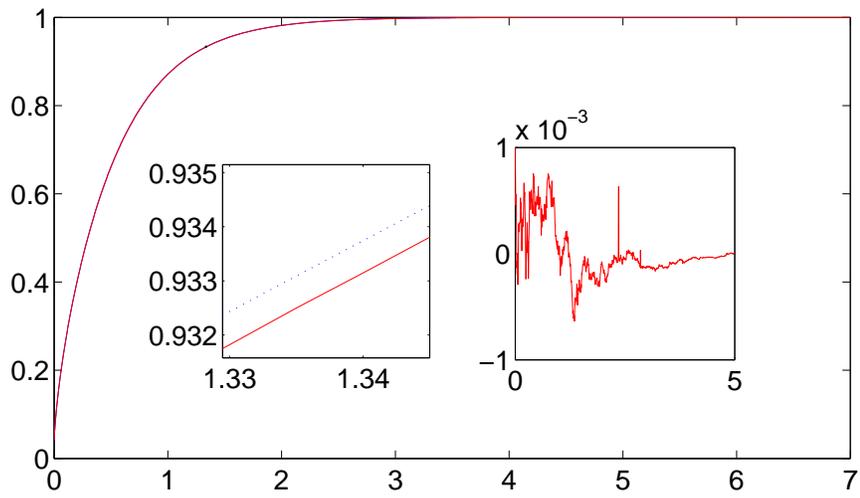}
\caption{\leftskip=25 pt \rightskip= 25 pt 
Comparison of simulation of right most excursion of the SAW in a strip and 
the conjectured distribution function using SLE. 
The simulation distribution is the dashed line.
}
\label{right_most}
\end{figure}

Figure \ref{right_most} shows the comparison of this prediction and 
our simulations. The two curves are indistinguishable on the scale 
of the main figure. The left inset shows a blow up of a small section 
where they deviate the most. (The dashed line is the 
simulation distribution.)
The right inset shows a plot of the difference
of the two curves. This difference is less than $1/10$ of a percent. 

\section{Conclusions}

We have reviewed the construction of the probability measure on 
infinite SAW's in the half plane as an i.i.d. sequence of irreducible 
bridges and the proof that this measure is the weak limit as 
$N \ra \infty$ of the uniform probability measure on SAW's with $N$ steps. 
We have shown that the measure on infinite length SAW's is the 
weak limit as $\beta \ra \beta_c-$ of the probability 
measure on all finite length SAW's in which the probability of 
a SAW $\omega$ is proportional to $\beta^{|\omega|}$.

We have considered the SAW in a strip of height $y$ which starts at the 
origin and ends anywhere on the upper boundary. 
The probability measure in which the probability of $\omega$ is proportional to 
$\beta_c^{|\omega|}$ can also be obtained by taking the SAW in the half plane 
and conditioning on the event that $y$ is a bridge height for the SAW.  

Using this relationship we have carried out simulations of the SAW 
in the strip and found good agreement with the
conjecture of Lawler, Schramm and Werner for the density of 
the endpoint of the SAW along the upper boundary. This is the 
first test of their prediction for such boundary densities. 
Our simulations have also tested the conjecture that the scaling 
limit of the SAW is SLE$_{8/3}$ for the strip and found good agreement.  

Our simulations have provided the first test of the conjectured 
conformal covariance of the boundary density, but the strip is a
rather special test case since the boundary is always parallel to a
lattice direction. It is expected that even in the scaling limit
there will be lattice effects that must be taken into account 
for other domains. This merits further study.

\end{document}